\documentclass[12pt,a4paper]{article}

\usepackage{amsmath}

\usepackage{latexsym,bm}
\usepackage{amsfonts,amsthm,amssymb,bbding,mathrsfs}
\usepackage{indentfirst}
\usepackage{graphicx}
\usepackage{color}
\usepackage[colorlinks=true,anchorcolor=blue,filecolor=blue,linkcolor=blue,urlcolor=blue,citecolor=blue]{hyperref}
\usepackage{float}
\usepackage{tikz}

\evensidemargin=\oddsidemargin\topmargin=-1.5cm

\newtheorem{theorem}{Theorem}[section]

\newtheorem{lemma}{Lemma}[section]

\newtheorem{claim}{Claim}[section]

\newcommand{\ex}{\mathrm{ex}}

\addtocounter{section}{0}
\begin{document}


\title{ An improved upper bound for planar Tur\'an number of double star $S_{2,5}$}
\author{{\bf Xin Xu}, {\bf Yue Hu}, {\bf Xu Zhang}\\
\small  School of Sciences, North China University of Technology,\\
\small  Beijing $100144$, China \\}

\date{}
\maketitle
{\flushleft\large\bf Abstract}
The planar Tur\'{a}n number of a graph $H$, denoted by $\ex_{\mathcal{P}}(n,H)$, is the maximum number of edges in an $n$-vertex $H$-free planar graph. Recently, D. Ghosh, et al. initiated the topic of double stars and prove that $\ex_{\mathcal{P}}(n,S_{2,5})\leq \frac{20}{7}n$. In this paper, we continue to study this  and give a sharp upper bound $\ex_{\mathcal{P}}(n,S_{2,5})\leq \frac{19}{7}n-\frac{18}{7}$ for all $n\geq 1$, with equality when $n=12$. This improves Ghosh's result.
\begin{flushleft}
\end{flushleft}

\section{Introduction}
One of the famous results in extremal graph theory is Tur\'an Theorem \cite{Turan}, which gives the maximum number of edges in an $n$-vertex graph without containing  $K_{r}$ as a subgraph. The Erd$\ddot{\rm o}s$-Stone Theorem \cite{Erdos-Stone} extends this to the case for all non-bipartite graphs $H$ and shows that
$ex(n,H)=(1-\frac{1}{\chi(H) -1})\binom{n}{2}+o(n^{2})$, where $\chi(H)$ denotes the chromatic number of $H$. This latter result has been called
the `fundamental theorem of extremal graph theory'.
Over the last decade, a large quantity of research work has been carried out in Tur\'an-type problems of graphs and hypergraphs. More results can be found in \cite{Furedi_surv} and \cite{Keevash_surv}.

In 2016, Dowden \cite{Dowden_c4} initiated the study of Tur\'an-type problems when host graphs are planar, i.e., how many edges in an $n$-vertex planar graph without containing a given smaller graph. The {\it planar Tur\'an number} of a graph H, denoted by $\ex_{\mathcal{P}}(n,H)$, is the maximum number of edges in an $H$-free planar graph on $n$ vertices.
Dowden \cite{Dowden_c4} obtained the tight bound $\ex_{\mathcal{P}}(n,C_{4})\leq \frac{15}{7}(n-2)$ for all $n\geq 4$ and $\ex_{\mathcal{P}}(n,C_{5})\leq \frac{12n-33}{5}$ for all $n\geq 11$. It is interesting to note that Wang and Lih \cite{Wang-Lih} in 2007 gave some upper bounds on the size of $H$-free graphs that are $2$-cell embedded in a surface of nonnegative of Euler characteristic. They also obtained similar bounds for  $\ex_{\mathcal{P}}(n,C_{4})$, $\ex_{\mathcal{P}}(n,C_{5})$,$ex_{p}(n,C_{6})$.
Then Lan, Shi, Song \cite{Lan_theta} determined the upper bound $ex_{p}(n,\Theta_{4})\leq \frac{12}{7}(n-2)$ for all $n\geq 4$, $\ex_{\mathcal{P}}(n,\Theta_{5})\leq \frac{5}{2}(n-2)$ for all $n\geq 5$, and $\ex_{\mathcal{P}}(n,\Theta_{6})\leq \frac{18}{7}(n-2)$ for all $n\geq 7$, where $\Theta_{k}$ is obtained from a cycle $C_{k}$ by adding an additional edge joining any two non-consecutive vertices. They also demonstrated that the bounds for $\Theta_{4}$ and $\Theta_{5}$ are tight for infinitely many $n$.
Ghosh, Gy{\H{o}}ri, Martin, Paulos and Xiao \cite{Ghosh_c6} improved the bound for $ex_{p}(n,C_{6})$ showing that $\ex_{\mathcal{P}}(n,C_{6})\leq \frac{5}{2}n-7$ for all $n\geq 18$. Ghosh, Gy{\H{o}}ri, Paulos, Xiao and Zamora \cite{Ghosh_theta} also proved that $\ex_{\mathcal{P}}(n,\Theta_{6})\leq \frac{1}{7}(18n-48)$ for all $n\geq 14$. Planar Tur\'an number of other graphs, such as paths, matchings and wheels, are also considered. We refer the reader to the survey \cite{Lan_surv} for the results not mentioned here.

Recently, Gy{\H{o}}ri, Martin, Paulos and Xiao \cite{Ghosh_star} initiated the topic for double stars as the forbidden graph. An $(m,n)$-star, denoted by $S_{m,n}$, is the graph obtained from an edge $uv$ through joining the two end vertices with $m$ and $n$ vertices respectively. The edge $uv$ is called the {\it backbone} of the double star. They presented several results about $\ex_{\mathcal{P}}(n,S_{2,2})$, $\ex_{\mathcal{P}}(n,S_{2,3})$, $\ex_{\mathcal{P}}(n,S_{2,4})$, $\ex_{\mathcal{P}}(n,S_{2,5})$ and proved that $\frac{5}{2}n \leq \ex_{\mathcal{P}}(n,S_{2,5})\leq \frac{20}{7}n$. We improved the upper bound, which is tight for $n=12$. The following theorem is our main result.

\begin{theorem}\label{th1}
For any $n\geq 1$, $\ex_{\mathcal{P}}(n,S_{2,5})\leq \frac{19}{7}n-\frac{18}{7}$, with equality when $n=12$.
\end{theorem}

 We need to introduce more notations. The graphs considered here are simple and finite. Let $G=(V(G),E(G))$, where $V(G)$ and $E(G)$ are the vertex set and edge set respectively. Let $n=v(G)=|V(G)|$ and  $m=e(G)=|E(G)|$. We denote the degree of a vertex $v$ by $d(v)$, the minimum degree by $\delta(G)$ and maximum degree by $\Delta(G)$. Let $n_{k}$ be the number of vertices of degree $k$. Moreover, We use $N_{G}(v)$ to denote the set of vertices of $G$ adjacent to $v$. Let  $N_{G}[v]=N_{G}(v)\cup \{v\}$. For any subset $S\subset V(G)$, the subgraph induced on $S$ is denoted by $G[S]$. We denote by $G\backslash S$ the subgraph induced on $V(G)\backslash S$. If $S=\{v\}$, we simply write $G\backslash v$. We use $e[S,T]$ to denote the number of edges between $S$ and $T$, where $S$, $T$ are subsets of $V(G)$.
 A $k$-$l$ edge is an edge of which end vertices are of degree $m$ and $n$. A $k$-$l$-$s$-path is a path consisting of  three vertices of degree $k$, $l$ and $s$.
All terminology and notation not defined in this paper are the same as those in the book\cite{Bondy}.

The paper is organized as follows. We  prove the theorem in Section $2$. In Section $3$, some necessary lemmas that used in the proof of Theorem \ref{th1} are provided.

\section{Planar Tur\'an number of \texorpdfstring{$S_{2,5}$}{S2,5}}

Let $G$ be an $S_{2,5}$-free plane graph on $n$ vertices. We shall proceed the proof by induction on $n$.

The upper bound in Theorem \ref{th1} is tight for $n=12$. To see that, consider the $5$-regular maximal planar graph with $12$ vertices and $30$ edges, given in Figure \ref{fig_1}. This regular graph does not contain $S_{2,5}$, since there is a vertex of degree $6$ in $S_{2,5}$.

\begin{figure}[ht]
  \centering    \includegraphics[width=0.4\textwidth]{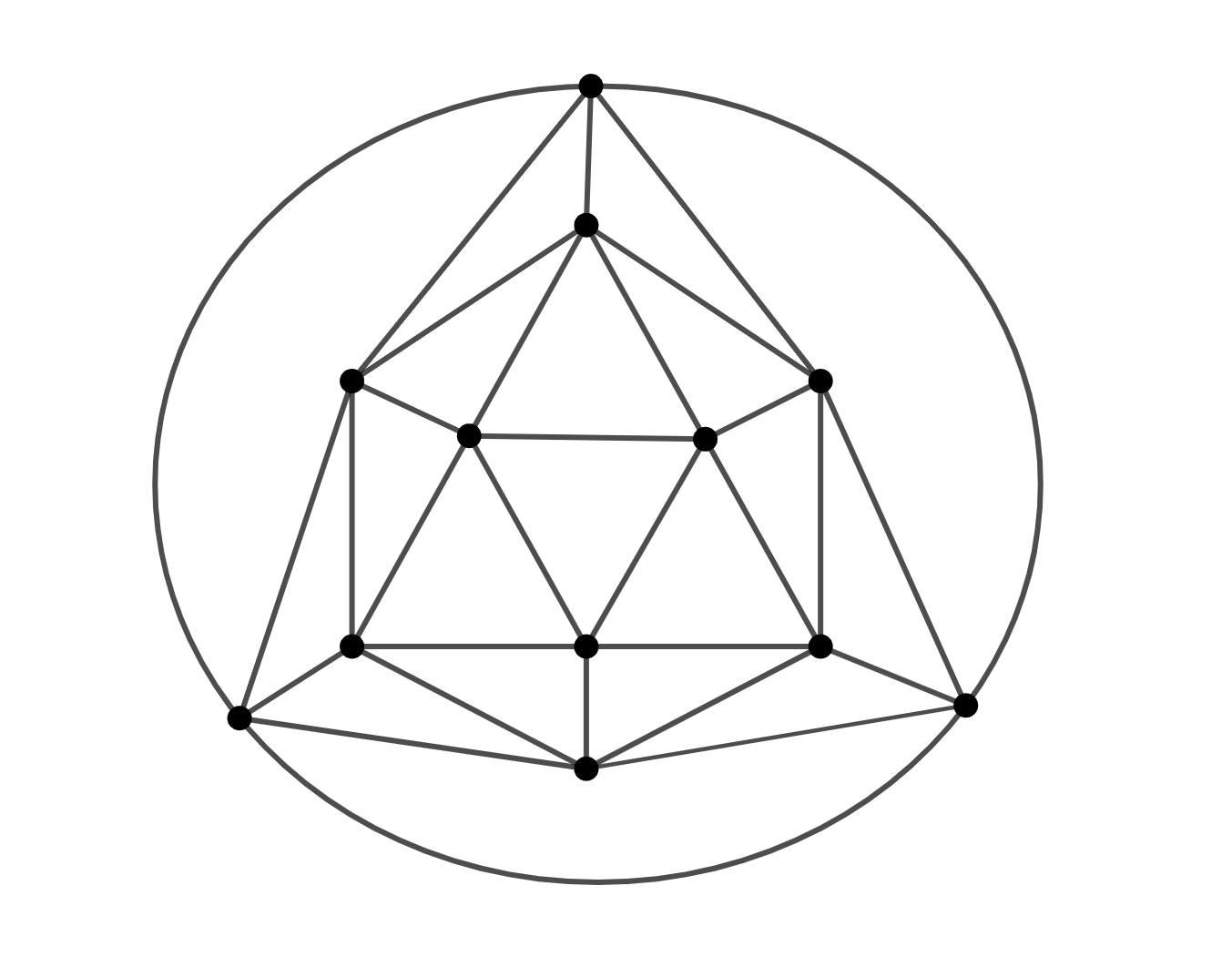}
  \caption{$5$-regular maximal planar graph on $12$ vertices.}
  \label{fig_1}
\end{figure}

\begin{claim}
If $G$ is a graph on  $n (1\leq n\leq 12)$ vertices. The number of edges in $G$ is at most $\frac{19}{7}n-\frac{18}{7}$.

\end{claim}
\begin{proof}
It is easy to see that a maximum planar graph with $n$ vertices contains $3n-6$ edges. Note that $3n-6 \leq \frac{19}{7}n-\frac{18}{7}$ for $n\leq 12$. We have $e(G)\leq \frac{19}{7}n-\frac{18}{7}$ when $1\leq n\leq 12$.
\end{proof}

It is worth  noting that if $G$ is disconnected, the  number of edges in each component satisfies the upper bound. Thus, $e(G)\leq \frac{19}{7}n-\frac{18}{7}$ by the induction hypothesis. So we may assume $G$ is connected.

\begin{lemma}\label{degree}
If $G$ has a vertex $v$ with $d(v)\leq 2$, then $e(G)\leq \frac{19}{7}n-\frac{18}{7}$.
\end{lemma}
\begin{proof}
If we delete the vertex $v$, then $e(G\backslash v)\leq \frac{19}{7}(n-1)-\frac{18}{7}$ by the induction hypothesis. Thus $e(G)= e(G\backslash v) + d(v)\leq \frac{19}{7}(n-1)-\frac{18}{7} +2 \leq \frac{19}{7}n-\frac{18}{7}$.
\end{proof}

\begin{lemma}\label{cutedge}
If $G$ has a cut edge, then $e(G)\leq \frac{19n}{7}-\frac{18}{7}$.
\end{lemma}
\begin{proof}
Assume that $uv$ is a cut edge and $H_{1}, H_{2}$ are the two components. Let $n'_{1}=|V(H_{1})|, n'_{2}=|V(H_{2})|$. Then
$e(H_{1})\leq \frac{19}{7}n'_{1}-\frac{18}{7}, e(H_{2})\leq \frac{19}{7}n'_{2}-\frac{18}{7}$ by the induction hypothesis.
We have $e(G)\leq e(H_{1})+ e(H_{2})+ 1 \leq \frac{19}{7}n-\frac{18}{7}$.
\end{proof}
Therefore,  we can suppose $\delta(G)\geq 3$ and $G$ does not contain a  cut edge.

\begin{claim}
If $\Delta(G)\geq 7$, then $e(G)\leq \frac{19}{7}n-\frac{18}{7}$.
\end{claim}
\begin{proof}
Recall that  each vertex $u \in N(v)$ has a degree greater that or equal to $3$. If there exists a vertex $v$  with $d(v)\geq 8$,   $G$ contains a copy of $S_{2,5}$, a contradiction.

Then let $v$ be a vertex with maximum degree $7$ and  $S=N_{G}[v]$. Note that if any $u\in N(v)$ is adjacent to a vertex in $G\backslash S$,  $G$ contains a copy of $S_{2,5}$ for $d(u)\geq 3$. So $G[S]$ is a connected component. Since $G$ is connected, we have $V(G)=S$. This implies $n=8$ and $e(G)\leq \frac{19}{7}n-\frac{18}{7}$.
\end{proof}

Next, we assume that $\Delta(G)\leq 6$.

\noindent{\bf Case 1.}\quad  $\Delta(G)\leq 5$.

Since there is a vertex of degree $6$ in $S_{2,5}$,  $G$ must be $S_{2,5}$-free. The maximum number of edges in $G$ is $\frac{5n}{2}\leq \frac{19}{7}n-\frac{18}{7}$ for $n\geq 12$.

\noindent{\bf Case 2.}\quad  $\Delta(G)=6$.

The following lemma is necessary for the proof, and we will prove it latter in Section $3$.

\begin{lemma}\label{lem}
Let $G$ be an  $n$-vertex  $S_{2,5}$-free connected graph without a cut edge. $\Delta(G)=6$ and $\delta(G)\geq 3$. If $G$ contains any copy of $6$-$6$ edge, $6$-$5$-$6$ path, $6$-$4$-$6$ path or $6$-$3$-$6$ path, then $e(G)\leq \frac{19}{7}n-\frac{18}{7}$.
\end{lemma}

Therefore, we assume that $G$ does not contain any subgraph mentioned in above lemma. Let $S=\{v\in V(G)| d(v)=6\}$ and $H=G[S]$.  Recall that $n_{k}$ is the number of vertices with degree $k$ and $n_{6}=|V(H)|$. Since there does not exist $6$-$6$ edge, $S$ is an independent set in $G$. Moreover, any two vertices in $H$ have no common neighbors. Otherwise, $G$ contains a $6$-$5$-$6$ path, $6$-$4$-$6$ path or $6$-$3$-$6$ path.

\begin{claim}\label{degree4}
Let $u$ be a vertex of degree $6$. There exits a vertex $v'\in N(u)$ such that $d(v')\leq 4$.
\end{claim}
\begin{proof}
Suppose that all vertices in $N(u)$ have degree $5$. Let $uv$ be a $6$-$5$ edge. Since $G$ is $S_{2,5}$-free, the edge $uv$ must be contained in $3$ or $4$ triangles.

\begin{figure}[ht]
  \centering
  \includegraphics[width=0.98\textwidth]{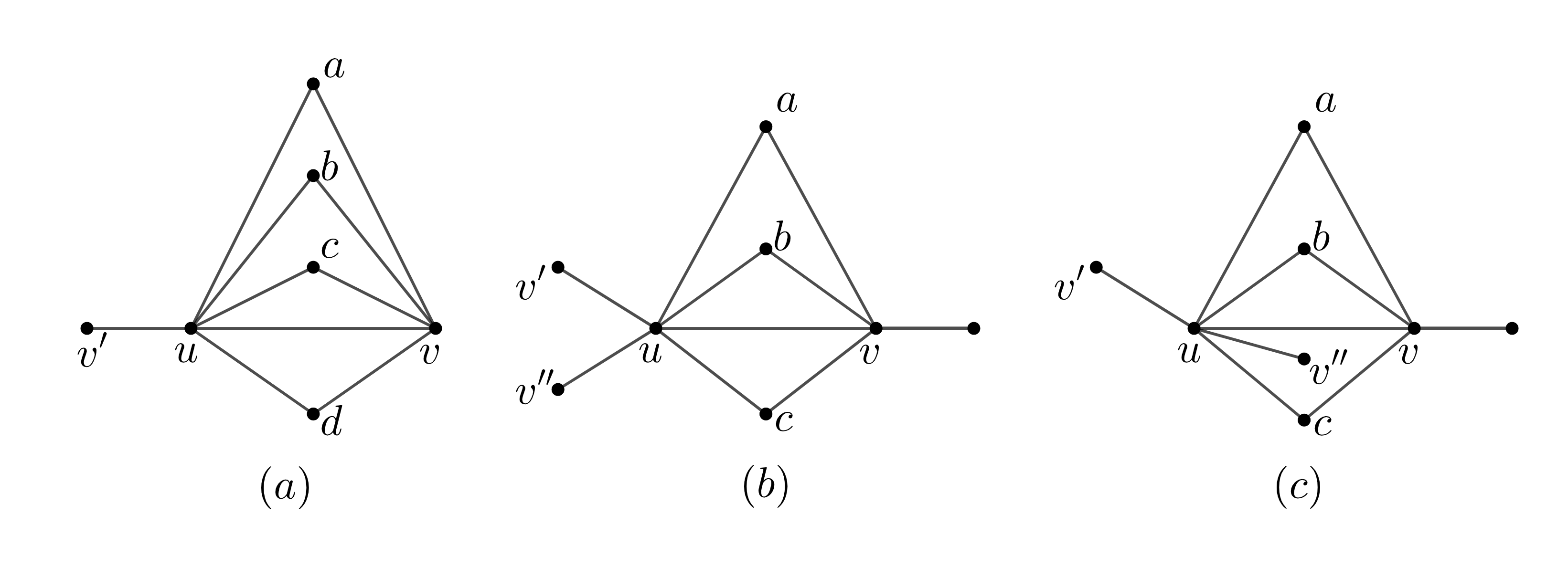}
  \caption{The graphs with a $6$-$5$ edge.}
  \label{fig_65}
\end{figure}

If there are $4$ triangles sitting on $uv$,  there  exists another vertex $v'$ adjacent to $u$  which is not adjacent to $v$, see Figure \ref{fig_65}$(a)$. In this plane graph, the vertex $v'$ can  be adjacent to at most two vertices in $N[u]$, such as $a$ and $d$. Since $d(v')=5$, the vertex $v'$ has two neighbors in $G\backslash N[u]$. We obtain that $G$ contains an $S_{2,5}$ and $uv'$ is the backbone, a contradiction.

If there are $3$ triangles sitting on $uv$,  there  exist two other vertices $v',v''$ adjacent to $u$  which are not adjacent to $v$, see Figure \ref{fig_65}$(b,c)$. If $v', v''$ are in the different faces of the plane graph, then $v'$ can  be adjacent to $a$ and $c$. Similarly, $v'$ has at least two neighbors in $G\backslash N[u]$ for $d(v')=5$, a contradiction.
Now we assume that $v', v''$ are in the same face of the plane graph.
If one vertex of $a$ or $c$ is not adjacent to $v'$, then  there are also at least two neighbors of $v'$ in $G\backslash N[u]$, a contradiction. Thus assume that  $a$ and $c$  are neighbors of $v'$. Then, the other vertex $v''$ can have at most one neighbor in $\{a, c\}$ and have at least two neighbors in $G\backslash N[u]$. This implies $G$ contains an $S_{2,5}$ and  $uv''$ is the backbone, a contradiction.
\end{proof}

Recall that any two vertices of degree $6$ have no common neighbor vertices and every vertex of degree $6$ has a neighbor of degree less than or equal to $4$.
We have $n_{6}\leq n_{3}+n_{4}$.

Since $n=\sum\limits_{i=3}^{6}n_{i}$,

\begin{align*}
m&=\frac{1}{2}\sum\limits_{v}d(v)=\frac{1}{2}\sum\limits_{i=3}^{6}i\cdot n_{i}\\
&=\frac{1}{2}(6\cdot n_{6}+4\cdot n_{4}+3\cdot n_{3}+5\cdot (n-n_{3}-n_{4}-n_{6}))\\
&=\frac{1}{2}(5\cdot n + n_{6}-n_{4}-2n_{3})\\
&\leq \frac{5}{2}n
\end{align*}

From the previous discussion, we obtain that $e(G)\leq \frac{5}{2}n \leq \frac{19}{7}n-\frac{18}{7}$ when $n\geq 12$. This completes the proof of Theorem \ref{th1}.

\section{Proof of Lemma \ref{lem}}

In this section, we  will prove the Lemma \ref{lem} used in the above. We also proceed the proof by induction on $n$. Thus, we suppose $\ex_{\mathcal{P}}(k,S_{2,5})\leq \frac{19}{7}k-\frac{18}{7}$ for $1\leq k \leq n-1$.
For convenience, we rewrite the lemma here.

\noindent{\bf Lemma \ref{lem}}
Let $G$ be an   $n$-vertex  $S_{2,5}$-free connected graph without a cut edge. $\Delta(G)=6$ and $\delta(G)\geq 3$. If $G$ contains any copy of $6$-$6$ edge, $6$-$5$-$6$ path, $6$-$4$-$6$ path or $6$-$3$-$6$ path, then $e(G)\leq \frac{19}{7}n-\frac{18}{7}$.

\begin{proof}
Based on the theorem above, we can assume that $n\geq 12$. Let us consider each case in turn.

\noindent{\bf Case 1.}\quad $G$ contains a $6$-$6$ edge.

Let $uv$ be the $6$-$6$ edge. Since $G$ is $S_{2,5}$-free, there are at least $4$ triangles sitting on the edge $uv$. We distinguish the cases based on the number of triangles sitting on $uv$.

\noindent{\bf Case 1.1.}\quad There are $5$ triangles sitting on $uv$.

\begin{figure}[ht]
  \centering

  \includegraphics[width=0.95\textwidth]{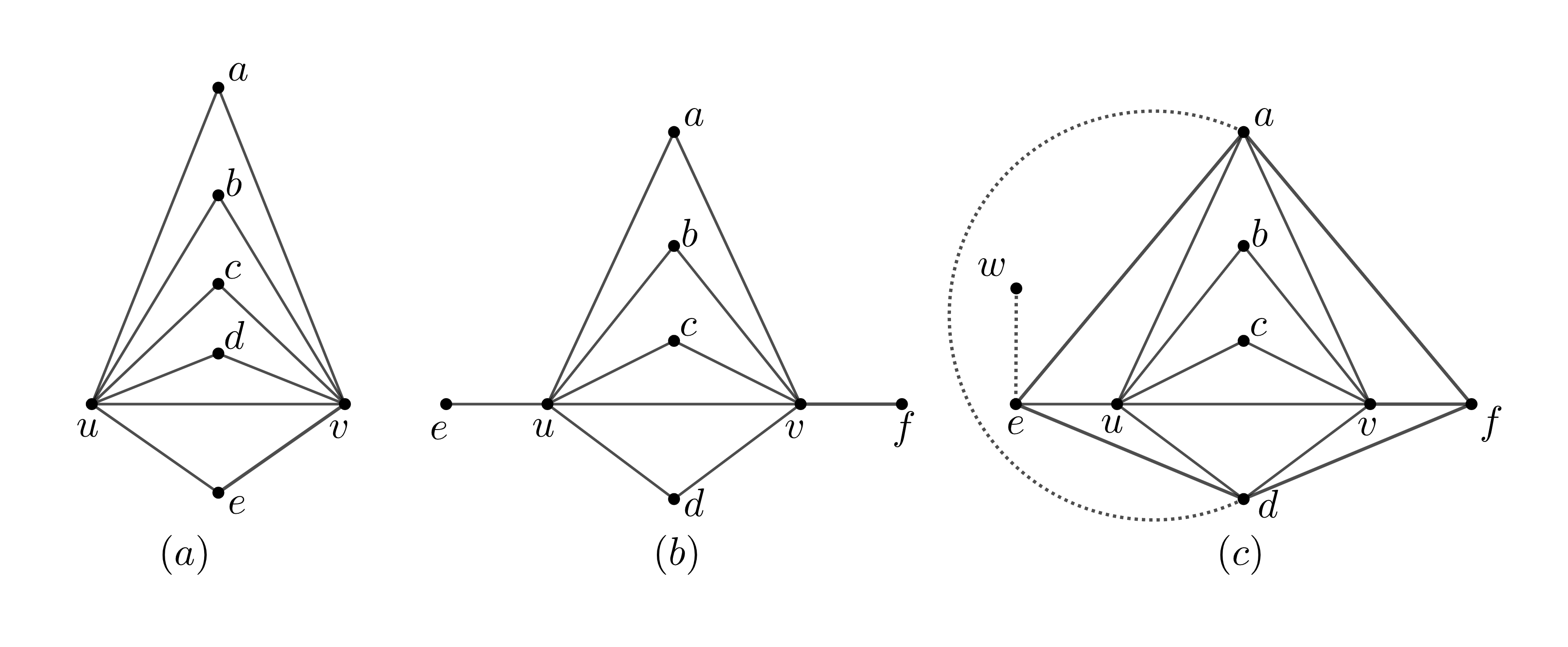}
  \caption{$(a), (b)$: The graphs with a $6$-$6$ edge $uv$ which $5$ or $4$ triangles sit on; $(c)$: $e, f$ have common neighbors in $S_{1}$}
  \label{fig_661}
\end{figure}

Let $a$, $b$, $c$, $d$ and $e$ be the vertices adjacent to both $u$ and $v$, see Figure \ref{fig_661}$(a)$. Let $S=\{u, v, a, b, c, d, e\}$, $S_{1}=\{a, b, c, d, e\}$ and $H=G[S]$. All vertices in $S_{1}$ can form a path of length at most $4$. We have $e(H)\leq 11+4=15$.

Note that all edges between $S$ and $V(G)\backslash S$ are those incident with vertices in $S_{1}$ and each vertex in $S_{1}$ has at most one neighbor in $V(G)\backslash S$.
Assume that all these $5$ edges exist. Since the cardinal number of $S_{1}$ is odd, there must exist a cut edge incident with some vertex in $S_{1}$, a contradiction.
Therefore, at most $4$ edges can be added between $H$ and $G\backslash H$.

By the induction hypothesis, $e(G)\leq e(G\backslash H)+ e(H)+ 4 = \frac{19}{7}(n-7)-\frac{18}{7}+19= \frac{19}{7}n-\frac{18}{7}$.

\noindent{\bf Case 1.2.}\quad There are $4$ triangles sitting on $uv$.

Let $a$, $b$, $c$ and $d$ be the vertices adjacent to both $u$ and $v$. Let $e$ be the vertex only adjacent to $u$ and $f$ be the vertex only adjacent to $v$, see Figure \ref{fig_661}$(b)$. Let $S=\{u, v, a, b, c, d, e, f\}$, $S_{1}=\{a, b, c, d\}$, $S_{2}=\{e, f\}$ and $S_{3}=\{u, v\}$. That means $S = S_{1}\cup S_{2}\cup S_{3}$. And we set $S'=V(G)\backslash S$, $H=G[S]$, $H'=G[S']$ and $H_{i}=G[S_{i}]$ for $i\in \{1, 2, 3\}$.

Then $e(G)= e(H)+e[S,S']+e(H')$,
where $e(H)= e(H_{1})+e(H_{2})+e(H_{3})+e[S_{1}, S_{2}]+e[S_{1}, S_{3}]+e[S_{2}, S_{3}]$ and $e[S,S']= e[S_{1}, S']+e[S_{2}, S']+e[S_{3}, S']$.
It is easy to know that $e(H_{3})+e[S_{1}, S_{3}]+e[S_{2}, S_{3}]=11$ and $e[S_{3}, S']=0$.
Thus $e(G)= e(H)+e[S,S']+e(H')=11+e(H')+e_{add}$, where $e_{add}=e(H_{1})+e[S_{1}, S']+e(H_{2})+e[S_{2}, S']+e[S_{1}, S_{2}]$.

Now we discuss the upper bound for $e_{add}$. Recall that $G$ is $S_{2,5}$-free. The vertices in $S_{1}$ may form a path of length at most $3$ and each vertex can have at most one neighbor in $H'$. Thus we have $e(H_{1})\leq 3$ and $e[S_{1}, S']\leq 4$. If  $e$ is adjacent to $f$, then $e(H_{2})=1$. And each vertex in $\{e, f\}$ can have at most one neighbor in $H'$, which means $e[S_{2}, S']\leq 2$. As $G$ is a plane graph, each vertex in $\{e, f\}$ (i.e. $S_{2}$) can have at most two neighbors in  $\{a, b, c, d\}$ (i.e. $S_{1}$). This means $e[S_{1}, S_{2}]\leq 4$. From the above results, it can be concluded that $e_{add}\leq 3+4+1+2+4=14$. Obviously, the discussion is a little rough. We need to improve the bound.

It is obtained that $e(H_{2})+e[S_{2}, S'] \leq 2$. Recall that $e(H_{2})\leq 1$ and $e[S_{2}, S'] \leq 2$.
In fact, if $e(H_{2})=1$, $H$ contains the edge $ef$. These two vertices each can not have a neighbor in $H'$, otherwise there exists an $S_{2,5}$. This implies $e[S_{2}, S']=0$.

Now, we show that  $e[S_{1}, S_{2}]+e[S_{1}, S']\leq 5$ in most cases.

\vspace{1mm}(i) If $e[S_{1}, S_{2}]=0$, then $e[S_{1}, S_{2}]+e[S_{1}, S']=e[S_{1}, S']\leq 4$.

\vspace{1mm}(ii) If $e[S_{1}, S_{2}]=1$, assume that $e$ has a neighbor $a$ in $S_{1}$ without loss of generality. Then the vertex $a$ can not have a neighbor in $H'$, which means $e[S_{1}, S']\leq 3$. So $e[S_{1}, S_{2}]+e[S_{1}, S']\leq 4$.

\vspace{1mm}(iii) If $e[S_{1}, S_{2}]=2$, there are two subcases. Assume that $e$ has two neighbors $a$, $d$, then $a, d$ can not have a neighbor in $H'$, meaning
$e[S_{1}, H']\leq 2$. Thus $e[S_{1}, S_{2}]+e[S_{1}, S']\leq 4$.
Now assume that  two vertices (i.e. $e$, $f$) each has a neighbor in $S_{1}$. If their neighbors are different, there exist two vertices in $S_{1}$
without a neighbor in $H'$. This implies $e[S_{1}, S']\leq 2$. Then $e[S_{1}, S_{2}]+e[S_{1}, S']\leq 4$. If  $e$, $f$ are adjacent to the same vertex $a$,
then only $a$ can not have a neighbor in $H'$. This means $e[S_{1}, S']\leq 3$ and $e[S_{1}, S_{2}]+e[S_{1}, S']\leq 5$.

\vspace{1mm}(iv) If $e[S_{1}, S_{2}]=3$, assume that $e$ has two neighbors $a$, $d$ and $f$ has one neighbor, denoted by $x$, in $S_{1}$. If the vertex $x \notin \{a, d\}$, these three vertices $a, d, x$ can not have a neighbor in $H'$. Thus $e[S_{1}, S_{2}]+e[S_{1}, S']\leq 3 + 1 =4$. Otherwise if $x\in \{a, d\}$, these two vertices $a, d$ can not have a neighbor in $H'$. We have $e[S_{1}, S_{2}]+e[S_{1}, S']\leq 3 + 2 =5$.

\vspace{1mm}(v) If $e[S_{1}, S_{2}]=4$, assume that $e$ has two neighbors $a$, $d$ and $f$ has two neighbors $x, y$.

1). If $|\{a, d\}\cup \{x,y\}|\geq 3$, there are at least three vertices which can not have a neighbor in $H'$. This means $e[S_{1}, S']\leq 1$ and $e[S_{1}, S_{2}]+e[S_{1}, S']\leq 5$.

2). If $ \{a, d\} = \{x,y\}$, there are two vertices which can not have a neighbor in $H'$. This implies $e[S_{1}, S']\leq 2$ and $e[S_{1}, S_{2}]+e[S_{1}, S']\leq 6$.

Next we will prove that $e_{add}\leq 10$.

It is easy to know that $e_{add}\leq 10$ when $e[S_{1}, S_{2}]+e[S_{1}, S']\leq 5$. So the only case we need to analysis is the second subcase of (v). That is  $e, f$ have common neighbors in $S_{1}$, see Figure \ref{fig_661}$(c)$.

If $e(H_{1})=3$, $a, b, c, d$ form a path of length $3$. Since $ad$ is an edge, $e$ is an interior vertex in the triangle $adu$. However, we claim that $e$ and $f$ can not have a neighbor in $H'$. Recall that the vertices $u, a, d$  do not have a neighbor in $H'$.
In fact, if there is the edge $ew$ with $w\in H'$,  it is easy to check that $ew$ is a cut edge in $G$. Similarly we can prove that for $f$. So  $e[S_{2}, S']=0$ and $e_{add}= e(H_{1})+(e[S_{1}, S_{2}]+e[S_{1}, S'])+(e(H_{2})+e[S_{2}, S'])\leq 3 + 6 + 1 \leq 10$.

If  $e(H_{1})\leq 2$, then $e_{add}\leq 2 + 6 + 2 =10$.

From the above discussion, we have $e_{add}\leq 10$.

Thus $e(G)= e(H)+e[S,S']+e(H')=11+e(H')+e_{add}\leq 21 + \frac{19}{7}(n-8)-\frac{18}{7}\leq \frac{19}{7}n-\frac{18}{7}$.

With loss of generality, we suppose that $G$ does not contain a $6$-$6$ edge. This implies that vertices of degree $6$ in $G$ are not adjacent to each other.

\noindent{\bf Case 2.}\quad $G$ contains a $6$-$5$-$6$ path.

Let $u, v, w$ be the vertices in the $6$-$5$-$6$ path and $d(v)=5$. We know that $uv$ and $wv$ are  $6$-$5$ edges. There are  $3$ or $4$ triangles sitting on the $6$-$5$ edge, otherwise $G$ contains an $S_{2,5}$. As shown in Figure \ref{fig_65}, these two graphs have the $6$-$5$ edge $uv$.

If there are $4$ triangles sitting on $uv$, see Figure \ref{fig_65}$(a)$, it is easy to see that the vertex $w$ will be a neighbor of $u$, a contradiction.

Thus there are $3$ triangles sitting on $uv$, see Figure \ref{fig_65}$(b)$. But it is confirmed that $w$ and $v$ have at most two common neighbor vertices, which implies there are at most two triangles sitting on the $6$-$5$ edge $wv$. Then we find an $S_{2,5}$, a contradiction.

Therefore, if $G$ contains a $6$-$5$-$6$ path, there must exist a $6$-$6$ edge. We have $e(G)\leq \frac{19}{7}n-\frac{18}{7}$.

\begin{figure}[ht]
  \centering
  \includegraphics[width=0.95\textwidth]{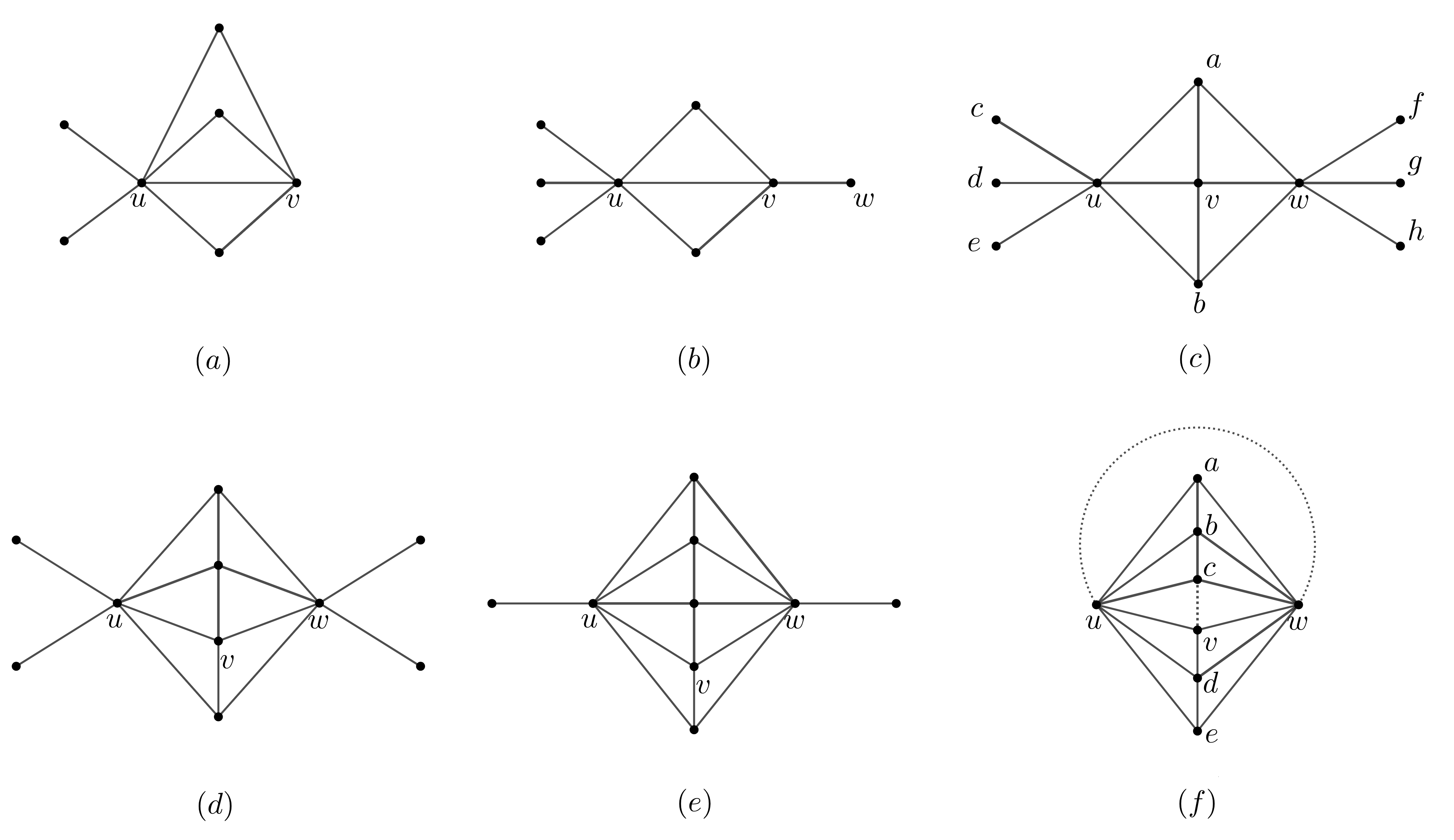}
  \caption{The graphs with a $6$-$4$ edge and a $6$-$4$-$6$ path.}
  \label{fig_64}
\end{figure}

\noindent{\bf Case 3.}\quad $G$ contains a $6$-$4$-$6$ path.

Let $u, v, w$ be the vertices in the $6$-$4$-$6$ path and $d(v)=4$, where $uv$ and $wv$ are $6$-$4$ edges. There  are $2$ or $3$ triangles sitting on the $6$-$4$ edge, otherwise an $S_{2,5}$ is contained.

If there are $3$ triangles sitting on  $uv$, see Figure \ref{fig_64}$(a)$, the vertex $w$ must be adjacent to $u$. We find a $6$-$6$ edge $uw$, a contradiction.

Thus there are $2$ triangles sitting on the $6$-$4$ edge, see Figure \ref{fig_64}$(b)$.
We claim that $G$ does not contain a $3$-$3$ edge. Indeed, if there is a $3$-$3$ edge, then we have $e(G)\leq \frac{19}{7}(n-2)-\frac{18}{7}+5\leq \frac{19}{7}n-\frac{18}{7}$ by induction hypothesis.

Then the vertex $w$  can be determined  uniquely and  four $6$-$4$-$6$ paths without $6$-$6$ edge or $3$-$3$ edge are constructed, shown in Figure \ref{fig_64}.

Next we discuss the subcase $(c)$ firstly.
Let $S=\{u, v, w, a, b, c, d, e, f, g, h\}$, $S_{1}=\{a, b\}$, $S_{2}=\{c, d, e, f, g, h\}$ and $S_{3}=\{u, v, w\}$. Let $S'=V(G)\backslash S$, $H=G[S]$, $H'=G[S']$ and $H_{i}=G[S_{i}]$ for $i\in \{1, 2, 3\}$.

Note that $u, v, w$ has no other neighbors in $G$. Similarly, $e(G)= e(H)+e[S,S']+e(H')=14+e(H')+e_{add}$, where $e_{add}=e(H_{1})+e[S_{1}, S']+e[S_{1}, S_{2}]+e(H_{2})+e[S_{2}, S']$.

Since $S_{1}=\{a, b\}$, we have $e(H_{1})\leq 1$. And each vertex in $\{a, b\}$ has no other neighbors in $G\backslash  S_{3}$, otherwise, there exists an $S_{2,5}$. This implies $e[S_{1}, S']=0$ and $e[S_{1}, S_{2}]=0$. 
Note that each vertex in   $\{c, d, e\}$ can have at most one vertex in  $ \{f, g, h\}\cup S'$, otherwise an $S_{2,5}$ is found. Similarly,  each vertex in $\{f, g, h\}$ can have at most one vertex in  $\{c, d, e\}\cup S'$. Furthermore, $\{c, d, e\}$  can form a triangle. It is the same to $\{f, g, h\}$.  So $e(H_{2})+e[S_{2}, S']\leq 12$.

From the above, $e_{add}\leq 1+12=13$. Thus by induction hypothesis $e(G)= 14+e(H')+e_{add}\leq 14 + \frac{19}{7}(n-11)-\frac{18}{7} + 13\leq \frac{19}{7}n-\frac{18}{7}$.

The proofs are similar to the subcases $(d), (e)$ respectively. We  skip it here for conciseness.

We consider the subcase $(f)$. Let $S=\{u, v, w, a, b, c, d, e\}$. Note that any vertex in $S$ can not have a neighbor in $G\backslash S$, otherwise an $S_{2,5}$ is found. Since $G$ is connected and $|S|=8$, we have $|V(G)|=8$. Indeed, if $|G|\geq 9$, it is easy to see that $G$ is disconnected. Thus $e(G)\leq 18\leq \frac{19}{7}\cdot 8-\frac{18}{7}$.

\noindent{\bf Case 4.}\quad $G$ contains a $6$-$3$-$6$ path.

Let $u, v, w$ be the vertices in the $6$-$3$-$6$ path and $d(v)=3$, where $uv$ and $wv$ are $6$-$3$ edges. There  are $1$ or $2$ triangles sitting on the $6$-$3$ edge, based on $G$ is $S_{2,5}$-free.

\begin{figure}[ht]
  \centering
  \includegraphics[width=0.9\textwidth]{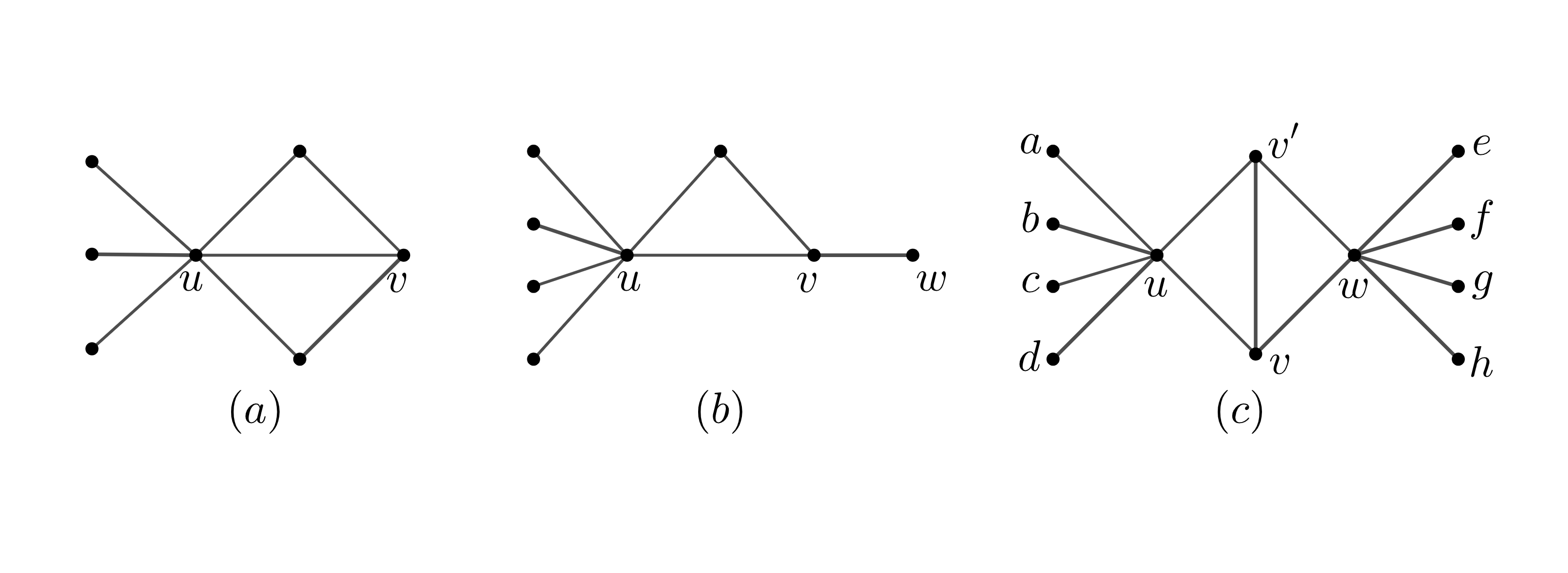}
  \caption{The graphs with a $6$-$3$ edge and a $6$-$3$-$6$ path.}
  \label{fig_63}
\end{figure}

If there are $2$ triangles sitting on $uv$, see Figure \ref{fig_63}$(a)$,  $uw$ is an edge, which contradicts with the fact that the vertices of degree $6$ in $G$ are not adjacent to each other.

Thus there is only one triangle sitting on $uv$, see Figure \ref{fig_63}$(b)$, we add $w$ to the $6$-$3$ edge $uv$ in order to form a $6$-$3$-$6$ path without  $6$-$6$ edge, $6$-$5$-$6$ path or $6$-$4$-$6$ path.

Suppose that $\{a, b, c, d\} \cap \{e, f, g, h\}= \emptyset$, shown in Figure \ref{fig_63}$(c)$. It is noticed that the vertex $v'$ can not have a neighbor in $G\backslash \{u,v,w\}$, otherwise $G$ has an $S_{2,5}$. This means $d(v')=3$. Then $vv'$ is the $3$-$3$ edge. Recall that $G$ has no $3$-$3$ edge, a contradiction.

On the other hand, if $\{a, b, c, d\} \cap \{e, f, g, h\}\neq \emptyset$, it is easy to find  a $3$-$3$ edge or $6$-$4$-$6$ path in $G$.

So we have $e(G)\leq \frac{19}{7}n-\frac{18}{7}$.

This completes the proof of the lemma.
\end{proof}

\section{Acknowledgement}
This work was supported by Research Funds (No. 2023YZZKY19) and Cultivation Plan for ``Yujie" Team (No. 107051360022XN725) of North China University of Technology.

\end{document}